\documentclass[a4paper,leqno,11pt]{amsart}
\usepackage{amsfonts,amssymb,verbatim,amsmath,amsthm,latexsym,textcomp,amscd}
\usepackage{latexsym,amsfonts,amssymb,epsfig,verbatim}
\usepackage{amsmath,amsthm,amssymb,latexsym,graphics,textcomp}
\usepackage{graphicx}
\usepackage{color}
\usepackage{setspace,cite}
\usepackage{url}
\usepackage{enumerate}
\usepackage{tikz-cd}
\usetikzlibrary{arrows}
\usepackage[mathscr]{euscript}
\input xy
\xyoption{all}

\usepackage{hyperref}

\theoremstyle{plain}
\newtheorem{theorem}{Theorem}[section]
\newtheorem{thm}[theorem]{Theorem}
\newtheorem{lemma}[theorem]{Lemma}
\newtheorem{cor}[theorem]{Corollary}
\newtheorem{prop}[theorem]{Proposition}

\newtheorem{example}[theorem]{Example}
\newtheorem{remark}[theorem]{Remark}

\theoremstyle{definition}
\newtheorem{defn}[theorem]{Definition}
\newtheorem{rmk}[theorem]{Remark}
\newtheorem{exam}[theorem]{Example}
\begin{document}
\title{Equivariant one-parameter formal deformations of Hom-Leibniz algebras}
\author{Goutam Mukherjee}
\email{gmukherjee.isi@gmail.com}
\address{Stat-Math Unit,
 Indian Statistical Institute, Kolkata 700108,
West Bengal, India.}

\author{Ripan Saha}
\email{ripanjumaths@gmail.com}
\address{Department of Mathematics,
Raiganj University, Raiganj, 733134,
West Bengal, India.}

\begin{abstract}
Aim of this paper is to  define a new type of cohomology for multiplicative Hom-Leibniz algebras which controls deformations of Hom-Leibniz algebra structure. The cohomology and the associated deformation theory for  Hom-Leibniz algebras as developed here are also extended to equvariant context, under the presence of finite group actions on Hom-Leibniz algebras.
\end{abstract}
\maketitle
\section{Introduction}

Gerstenhaber in a series of papers \cite{G1, G2, G3, G4, G5} introduced the notion of algebraic deformation theory for associative algebras. Later following Gerstenhaber, deformation theory of other algebraic structures are studied extensively in various context ( \cite{NR}, \cite{fox}, \cite{GS}, \cite{MM}, \cite{NR66}). For example, A. Nijenhuis and R. Richardson studied formal one-parameter deformation theory of Lie algebras \cite{NR}.

To study deformation theory of a type of algebra one needs a suitable cohomology, called deformation cohomology which controls deformations in question. In the case of associative algebras, deformation cohomology is Hochschild cohomology and for Lie algebras, the associated deformation cohomology is Chevalley-Eilenberg cohomology.

Hartwig, Larsson, and Silvestrov introduced the notion of Hom-Lie algebras in \cite{HLS}. Hom-Lie algebras appeared as examples of $q$-deformations of the Witt and Virasoro algebras. In \cite{MZ}, Makhlouf and Zusmanovich described Hom-Lie algebra structures on affine Kac-Moody algebras.  The notion of a Hom-Lie algebra structure is a generalization of Lie algebra structure on a vector space. A Lie algebra equipped with self linear map, called the structure map,  is said to be a  Hom-Lie algebra if the Jacobi identity is replaced by Hom-Jacobi identity, which is Jacobi identity twisted by the structure map. Obviously, a Hom-Lie algebra is a Lie algebra when the associated structure map is identity.

In \cite{loday93}, J.-L. Loday introduced a version of non anti-symmetric Lie algebra and its (co)homology, known as Leibniz algebra. The bracket of a Leibniz algebra satisfies Leibniz identity instead of Jacobi identity. In the presence of skew-symmetry Leibniz identity reduces to Jacobi identity. Cohomology of any Leibniz algebra with coefficients in a bimodule was introduced in  \cite{LP}.

Makhlouf and Silvestrov introduced the notion of a Hom-Leibniz algebra in \cite{MS08} generalizing Hom-Lie algebras. Thus, Hom-Leibniz algebras are generalizations of both Leibniz and Hom-Lie algebras. In Hom-Leibniz algebras, Leibniz identity is twisted by a self linear map and it is called Hom-Leibniz identity. Other variants of Hom-type algebras have been studied in \cite{AEM}, \cite{M10}, \cite{MS08}, \cite{MS10}, \cite{MS102}, \cite{saha}.

In \cite{MS10}, Makhlouf and Silvestrov introduced deformation cohomologies of first and second order to study one-parameter formal deformation theory for Hom-associative and Hom-Lie algebras. Hurle and Makhlouf introduced a new type of cohomology theory considering the structure map for Hom-associative and Hom-Lie algebras in \cite{HM19}, \cite{HM19glas}. Cheng and Cai defined cohomology groups of all orders for Hom-Leibniz algebras in \cite{CA}. 

In the present paper, we define a cohomology theory for multiplicative Hom-Leibniz algebras generalizing \cite{CA}. We call this new cohomology as $\alpha$-type cohomology for Hom-Leibniz algebras. We also develop one-parameter formal deformation theory for Hom-Leibniz algebras using $\alpha$-type cohomology as the deformation cohomology. 

Finally, we define a notion of finite group action on Hom-Leibniz algebras along the line of Bredon cohomology of a G-space, \cite{bredon67} and define equivariant version of $\alpha$-type cohomology. It turns out that for a Hom-Leibniz algebra equipped with an action of a finite group, its equivariant deformations are controlled by this $\alpha$-type cohomology. Note that an action of a finite group $G$ on a Hom-Leibniz algebra $L$ over a field $\mathbb K$ naturally extends to the  formal power series $L[[t]]$ by bilinearly extending the multiplication of $L$ making  $L[[t]]$ a Hom- Leibniz algebra over $\mathbb K [[t]].$ 

The paper is organized as follows. In Section \ref{sec 1}, we recall basics of Hom-Leibniz algebras which we shall use throughout the paper.
In Section \ref{sec 2}, we show that there is a Gerstenhaber bracket on shifted cochains for Hom-Leibniz cohomology introduced in \cite{CA}, and the bracket induces a graded Lie algebra structure on the graded cohomology. In Section \ref{sec 3}, we introduce $\alpha$-type cohomology of multiplicative Hom-Leibniz algebras. In Section \ref{sec 4}, we introduce one-parameter formal deformation theory of Hom-Leibniz algebras. We define infinitesimal deformation, study the the problem of extending a given deformation of order $n$ to a deformation of order $(n+1)$ and define the associated obstruction. We also study rigidity conditions for formal deformations. In Section \ref{sec 5}, we define the notion of finite group actions on Hom-Leibniz algebras and introduce equivariant  $\alpha$-type cohomology for Hom-Leibniz algebras equipped with a finite group action. In the final Section \ref{sec 6}, we define equivariant formal deformations and prove that equivariant  $\alpha$-type cohomology is the right notion of deformation cohomology in the present context. We end with a brief discussion of rigidity of equivariant deformations for Hom-Leibniz algebras equipped with finite group action.

\section{Preliminaries}\label{sec 1}
In this section, we recall the basics of Hom-Leibniz algebras. Let $\mathbb{K}$ be a field of characteristic zero. Though most of the constructions should also work in other characteristics (not $2$) or if $\mathbb{K}$ is a ring containing the rational numbers.
\begin{defn}
A Hom-Leibniz algebra is a $\mathbb{K}$-linear vector space $L$ together with a $\mathbb{K}$-bilinear map $[.,.]:L\times L\to L$ and a $\mathbb{K}$-linear map (structure map) $\alpha:L\to L$ satisfying Hom-Leibniz identity:
\begin{center}
$[\alpha(x),[y,z]]=[[x,y],\alpha(z)]-[[x,z],\alpha(y)].$
\end{center}
\end{defn}
A Hom-Leibniz algebra $(L_1,[.,.],\alpha)$ is called multiplicative if $\mathbb{K}$-linear map $\alpha$ satisfies $\alpha ([x,y])=[\alpha(x),\alpha(y)]$.

A morphism between Hom-Leibniz algebras $(L_1,[.,.]_1,\alpha_1)$ and $(L_2,[.,.]_2,\alpha_2)$ is a $\mathbb{K}$-linear map $\phi: L_1\to L_2$ which satisfies $\phi([x,y]_1)=[\phi(x),\phi(y)]_2$ and $\phi\circ\alpha_1=\alpha_2\circ\phi$.
\begin{example}
Any Hom-Lie algebra is automatically a Hom-Leibniz algebra as in the presence of skew-symmetry Hom-Leibniz identity is same as Hom-Jacobi identity.
\end{example}
\begin{example}
Given a Leibniz algebra $(L,[.,.])$ and a Leibniz algebra morphism $\alpha: L\to L$, one always get a Hom-Leibniz algebra $(L,[.,.]_\alpha,\alpha)$, where $[x, y]_\alpha= [\alpha(x), \alpha(y)]$.
\end{example}
\begin{example}
Let $L$ is a two-dimensional $\mathbb{C}$-vector space with basis $\lbrace e_1, e_2\rbrace$. We define a bracket operation as $[e_2,e_2]=e_1$ and zero else where and the endomorphism is given by the matrix \[
\alpha=
  \begin{bmatrix}
    1 & 1  \\
    0 & 1 
  \end{bmatrix}.
\]
It is a routine work to check that $(L, [.,.],\alpha)$ is a Hom-Leibniz algebra which is not Hom-Lie.
\end{example}
\begin{defn}
A Hom-vector space is a $\mathbb{K}$-vector space $M$ together with a $\mathbb{K}$-linear map $\beta:M\to M$ such that vector space operations are compatible with $\beta$. We write a Hom-vector space as $(M,\beta)$.
\end{defn}
 \begin{defn}
 Let $(L,[.,.],\alpha)$ be a Hom-Leibniz algebra. A  $L$-bimodule is a Hom-vector space  $(M,\beta)$ together with two $L$-actions (left and right multiplications), $m_l:L\otimes M\to M $ and $m_r:M\otimes L\to M $ satisfying the following conditions,
 \begin{align*}
 &\beta (m_l (x, m)) = m_l(\alpha(x), \beta(m)),\\
 &\beta (m_r (m, x)) = m_r(\beta(m), \alpha(x)),\\
 &m_r (\beta(m), [x,y]) = m_r(m_r(m, x), \alpha(y)) - m_r (m_r(m, y), \alpha(x)),\\
 & m_l (\alpha(x), m_r(m, y))=  m_r(m_l(x, m), \alpha (y))- m_l([x, y], \beta(m)),\\
 & m_l(\alpha(x), m_l(y, m)) = m_l([x, y], \beta(m)) - m_r(m_l(x, m), y),
 \end{align*}
 for any $x, y \in L$ and $m\in M$.
 \end{defn}
Note that any Hom-Leibniz algebra $(L,[.,.],\alpha)$ can be considered as a bimodule over itself by taking $m_l=m_r=[.,.]$ and $\beta=\alpha$.

We recall the cohomology of Hom-Leibniz algebra $(L,[.,.],\alpha)$ defined in \cite{CA}.
Let 
$${CL}_\alpha^n(L, L) = \lbrace \phi : L^{\otimes n} \to L \mid \alpha \circ \phi = \phi \circ \alpha^{\otimes n}\rbrace.$$
For $n\geq 1$,  $\delta^n:CL^n_{\alpha}(L,L)\to CL^{n+1}_{\alpha}(L,L)$ is defined as follows:
\begin{align}
	&(\delta^n \phi)(x_1,\dots,x_{n+1}) \\ \nonumber
	& =[\alpha^{n-1}(x_1),\phi(x_2,\ldots,x_{n+1})]\\\nonumber
 &+\sum^{n+1}_{i=2}(-1)^{i}[\phi(x_1,\ldots,\hat{x_i},\ldots,x_{n+1}),\alpha^{n-1}(x_i)]                                    \\
	                                 & +\sum_{1\leq i<j\leq n+1}(-1)^{j+1}\phi(\alpha(x_1),\ldots,\alpha(x_{i-1}),[x_i,x_j],\alpha(x_{i+1}),\ldots,\widehat{\alpha(x_j)},\ldots,\alpha(x_{n+1})).   \nonumber                
\end{align}	                                 
For $n\geq 1$, $\delta^2=0$ and $({CL}_\alpha^\ast(L, L), \delta)$ is a cochain complex. The cohomology of this cochain complex is discussed in \cite{CA}.

\section{Gerstenhaber bracket on cochains for Hom-Leibniz Cohomology}\label{sec 2}
In \cite{AEM}, authors studied Gerstenhaber algebra structure on the shifted cochains for Hom-associative algebras. In this section, we define a Gerstenhaber bracket on shifted cochains for Hom-Leibniz algebra cohomology introduced by Cheng and Cai,\cite{CA} and show that this bracket induces a graded Lie algebra structure on cohomology of Hom-Leibniz algebra.

\begin{defn}
Let $S_n$ be the permutation group of $n$ elements $1, \ldots, n.$ A permutation $\sigma \in S_n$ is called a $(p, q)$-shuffle if $p+q = n$ and 
$$\sigma (1) < \cdots <\sigma (p) ~~\mbox{and}~~~\sigma (p+1) < \cdots <\sigma (p+q).$$ 
\end{defn}
In the group algebra $\mathbb K [S_n],$ let $Sh_{p,q}$ be the element $$Sh_{p,q}: = \sum_{\sigma} \sigma,$$ where the summation is over all $(p, q)\mbox{-shuffles}.$

Suppose for $n\geq-1$,\,$CH^n(L,L)$ is the space of all $(n+1)$-linear maps $\phi: L^{\otimes n+1}\to L$ satisfying
$$\alpha\circ \phi=\phi\circ \alpha^{\otimes n+1}.$$
Let $\phi\in CH^p(L,L)$ and $\psi\in CH^q(L,L)$, where $p\geq 0,q\geq 0$, we define $\psi\circ\phi\in CH^{p+q}(L,L)$ as follows,
\begin{align*}
\psi\circ\phi(x_1,\ldots,x_{p+q+1})=&\sum^{q+1}_{k=1}(-1)^{p(k-1)}\lbrace\sum_{\sigma\in Sh(p,q-k+1)}sgn(\sigma)\psi(\alpha^p(x_1),\ldots,\alpha^p(x_{k-1}),\\&\phi(x_k,x_{\sigma(k+1)},\ldots,x_{\sigma(k+p)}),\alpha^p(x_{\sigma(k+p+1)}),\ldots,\alpha^p(x_{\sigma(p+q+1)})\rbrace.
\end{align*}
Suppose
$$CH^*(L,L)=\bigoplus_{p\geq -1} CH^p(L,L).$$
We define a bracket $[.,.]$ on $CH^*(L,L)$ as $[\psi,\phi]=\psi\circ\phi+(-1)^{pq+1}\phi\circ\psi$. 
\begin{remark}
For $p=q=1$, 
\begin{align*}
& \psi\circ\phi(x_1,x_2,x_{3})\\
& =\psi(\alpha(x_1),\phi(x_2,x_3))-\psi(\phi(x_1,x_2),\alpha(x_3))+\psi(\phi(x_1,x_3),\alpha(x_2))=\psi\circ_{\alpha}\phi.
\end{align*}
For $\phi=\psi$, this is nothing but $\alpha$-associator of $\phi$.
\end{remark}
\begin{prop}
Suppose $m_1\in CH^1(L,L)$. Then $(L,m_1,\alpha)$ is a Hom-Leibniz algebra if and only if $[m_1,m_1]=0.$
\end{prop}
\begin{proof}
\begin{align*}
[m_1,m_1] (x_1, x_2, x_3)& =2\big(m_1(\alpha(x_1),m_1(x_2,x_3)) -m_1(m_1(x_1,x_2), \alpha(x_3))\\
&+m_1(m_1(x_1,x_3),\alpha(x_2))\big).
\end{align*}
Thus, $(L,m_1,\alpha)$ is a Hom-Leibniz algebra if and only if $[m_1,m_1]=0.$
\end{proof}
\begin{lemma}
Let $(L,m_0,\alpha)$ be a Hom-Leibniz algebra and $\phi\in CH^p(L,L)$, then $\delta\phi=-[\phi,m_0]$, where $m_0=[.,.]$ is the Leibniz bracket of $L$.
\end{lemma}
\begin{proof}
Let $x_1,x_1,\ldots,x_{p+2}\in L$. From the coboundary formula, we have
\begin{align*}
 & \delta\phi(x_1,\ldots,x_{p+2})\\
 &=[\alpha^{p}(x_1),\phi(x_2,\ldots,x_{p+2})]\\
 &+\sum^{p+2}_{i=2}(-1)^{i}[\phi(x_1,\ldots,\hat{x_i},\ldots,x_{p+2}),\alpha^{p}(x_i)]\\
&+\sum_{1\leq i<j\leq {p+2}}(-1)^{j+1}\phi(\alpha(x_1),\ldots,\alpha(x_{i-1}),[x_i,x_j],\alpha(x_{i+1}),\ldots,\widehat{\alpha(x_j)},\ldots,\alpha(x_{p+2})).
 \end{align*}
 Note that $m_0\in CH^1(L,L)$ and $[\phi,m_0]=\phi\circ m_0+(-1)^{p+1}m_0\circ \phi$. 
 \begin{align*}
\phi\circ m_0(x_1,\ldots,x_{p+2})=&\sum^{p+1}_{k=1}(-1)^{k-1}\lbrace\sum_{\sigma\in Sh(1,p-k+1)}sgn(\sigma)\phi(\alpha(x_1),\ldots,\alpha(x_{k-1}),\\& [x_k,x_{\sigma(k+1)}],\alpha(x_{\sigma(k+2)}),\ldots,\alpha(x_{\sigma(p+2)})\rbrace.\\
&=\sum_{1\leq k<j\leq (p+2)}(-1)^{j}\phi(\alpha(x_1),\ldots,\alpha(x_{k-1}),\\& [x_k,x_j],\alpha(x_{(k+1)}),\ldots,\alpha(x_{j-1}),\widehat{\alpha(x_j)},\alpha(x_{j+1}),\ldots,\alpha(x_{(p+2)})).
\end{align*}
 On the other hand, we have
 \begin{align*}
&m_0\circ\phi(x_1,\ldots,x_{p+2})\\
&=\sum_{\sigma\in Sh(p,1)}sgn(\sigma)m_0(\phi(x_1,x_{\sigma(1)},\ldots,x_{\sigma(p+1)}),\alpha^p(x_{\sigma(p+2)}))\\
&+(-1)^p[\alpha^p(x_1),\phi(x_2,\ldots,x_{p+2})]\\
&=\sum_{2\leq j\leq {p+2}}(-1)^{p+2-j}[\phi(x_1,x_2,\ldots,x_{j-1},\hat{x_j},x_{j+1},\ldots,x_{p+1}),\alpha^p(x_j)]\\
&+(-1)^p[\alpha^p(x_1),\phi(x_2,\ldots,x_{p+2})].
\end{align*}
Therefore,
\begin{align*}
[\phi,m_0] (x_1,\dots, x_{p+2}) & =(\phi\circ m_0+(-1)^{p+1}m_0\circ\phi)(x_1,\ldots,x_{p+2})\\
&=\sum_{1\leq k<j\leq p+2}(-1)^{j}\phi(\alpha(x_1), \ldots,\alpha(x_{k-1}), [x_k,x_j],\alpha(x_{k+1}),\ldots,\\
& \alpha(x_{j-1}),\widehat{\alpha(x_j)},
\alpha(x_{j+1}),\ldots,\alpha(x_{\sigma(p+2)})\\
&+\sum_{2\leq j\leq {p+2}}(-1)^{j+1}[\phi(x_1,x_2,\ldots,x_{j-1},\hat{x_j},x_{j+1},\ldots,x_{p+2}),\alpha^p(x_j)]\\
& -[\alpha^p(x_1),\phi(x_2,\ldots,x_{p+2})]
=-\delta\phi.
\end{align*}
Thus, we have $\delta\phi=-[\phi,m_0].$
\end{proof}
The graded $\mathbb{K}$-module $CH^*(L,L)=\bigoplus_{p\geq -1} CH^p(L,L)$ together with the bracket $[\psi,\phi]=\psi\circ\phi+(-1)^{pq+1}\phi\circ\psi$ is a graded Lie algebra \cite{bala}. If $\phi\in CH^p(L,L)$, then we define $|\phi|=p+1$.
We define a linear map $d:CH^*(L,L)\to CH^*(L,L)$ as follows:
$$d(\phi)=(-)^{|\phi|}\delta(\phi).$$
Using the graded Lie algebra structure on $CH^*(L,L)$, we prove the following lemma.
\begin{lemma}\label{differential}
The differential $d$ satisfies the the following graded derivation formula:
$$d[\psi,\phi]=[d\psi,\phi]+(-1)^{|\psi|}[\psi,d\phi],$$
for $\psi,\phi\in CH^*(L,L)$.
\end{lemma}
\begin{proof}
Let $\phi\in CH^p(L,L)$ and $\psi\in CH^q(L,L)$. To prove this Lemma, we use  properties of graded Lie algebra.
\begin{align*}
d[\psi,\phi]&=(-1)^{p+q+1}\delta[\psi,\phi]\\
                   &=-(-1)^{p+q+1}[[\psi,\phi],m_0]\\
                   &=[m_0,[\psi,\phi]]\\
                   &=-(-1)^{p(q+1)}[\phi,[m_0,\psi]]-(-1)^{p+q}[\psi,[\phi,m_0]]\\
                   &=(-1)^{q+1}[[m_0,\psi],\phi]+(-1)^q[\psi,d\phi]\\
                   &=[(-1)^q\delta\psi,\phi]+(-1)^{q+1}[\psi,d\phi]\\
                   &=[d\psi,\phi]+(-1)^{q+1}[\psi,d\phi]\\
                   &=[d\psi,\phi]+(-1)^{|\psi|}[\psi,d\phi].
\end{align*}
\end{proof}
 From the Lemma (\ref{differential}), $(CH^*(L,L),[.,.],d)$ is a differential graded Lie algebra. The cohomology group of $(CH^*(L,L),[.,.],d)$ is denoted by $(H^*_{HL}(L,L)$. It is clear from the Lemma (\ref{differential}) that the Gerstenhaber bracket on graded cochains induces a bracket $[.,.]$ on the cohomology level and we have the following theorem.
\begin{theorem}
$(H^*_{HL}(L,L),[.,.])$ is a graded Lie algebra.
\end{theorem}

\section{\texorpdfstring{$\alpha$}{alpha}-type Leibniz cohomology of Hom-Leibniz algebras}\label{sec 3}
For the deformation theory we need a new type of cohomology that can capture information of deformation for both multiplication and structure map of a Hom-Leibniz algebra.  We begin this section by introducing a cohomology theory for Hom-Leibniz algebras considering both the bracket and the structure map $\alpha$. We call this cohomology for Hom-Leibniz algebras by $\alpha$-type cohomology of Hom-Leibniz algebras. We will see that this cohomology is a generalization of the cohomology defined in the last section.

Let $\gamma(x, y) =[x, y]$ and we define the cochain complex for the cohomology of $(L, \gamma, \alpha)$ with coefficients in $L$ as follows:
\begin{align*}
\widetilde{CL}^n(L, L) & = \widetilde{CL}_\gamma^n(L, L) \oplus  \widetilde{CL}_\alpha^n(L, L)\\
                                      & = \text{Hom}_\mathbb{K}(L^{\otimes n}, L) \oplus \text{Hom}_\mathbb{K}(L^{\otimes {n-1}}, L),~~~\text{for all}~ n\geq 2.
\end{align*}
\begin{align*}
\widetilde{CL}^1 (L, L) & = \widetilde{CL}_\gamma^1 (L, L) \oplus  \widetilde{CL}_\alpha^1 (L, L) = \text{Hom}_\mathbb{K}(L, L) \oplus \lbrace 0 \rbrace.
\end{align*}
\begin{align*}
\widetilde{CL}^n(L, L)= \lbrace 0 \rbrace~~~\text{for all}~ n\leq 0.
\end{align*}

We may write elements of $\widetilde{CL}^n(L, L)$ as $(\phi, \psi)$ or $\phi + \psi$, where $\phi \in \widetilde{CL}_\gamma^n(L, L)$ and $\psi \in \widetilde{CL}_\alpha^n(L, L)$. Note that we set $\text{Hom}_\mathbb{K}(\mathbb{K}, L) = \lbrace 0 \rbrace$ instead of $\mathbb{K}$ as usual, otherwise $\alpha^{-1}$ would be needed in the definition of the differential.
We define four maps with domain and range given in the following diagram:
\[
\begin{tikzcd}
\widetilde{CL}_\gamma^n(L, L) \arrow[rdd, crossing over, "{\partial_{\gamma \alpha}}"  pos=.2] \arrow[r, "{\partial_{\gamma \gamma}}"] &  [10 ex] \widetilde{CL}_\gamma^{n+1}(L, L)\\ [-3 ex]
\bigoplus & \bigoplus\\ [-3 ex]
 \widetilde{CL}_\alpha^n(L, L)\arrow[uur,  "{\partial_{\alpha \gamma}}" pos=.2] \arrow[r, "{\partial_{\alpha \alpha}} "] & \widetilde{CL}_\alpha^{n+1}(L, L)
\end{tikzcd}
\]
\begin{align}
&(\partial_{\gamma \gamma} \phi)(x_1,\dots,x_{n+1}) \\ \nonumber
	& =[\alpha^{n-1}(x_1),\phi(x_2,\ldots,x_{n+1})]\\ \nonumber
 &+\sum^{n+1}_{i=2}(-1)^{i}[\phi(x_1,\ldots,\hat{x_i},\ldots,x_{n+1}),\alpha^{n-1}(x_i)]                                    \\
	                                 & +\sum_{1\leq i<j\leq n+1}(-1)^{j+1}\phi(\alpha(x_1),\ldots,\alpha(x_{i-1}),[x_i,x_j],\alpha(x_{i+1}),\ldots,\widehat{\alpha(x_j)},\ldots,\alpha(x_{n+1})).   \nonumber 
 \end{align}
 
\begin{align}
&(\partial_{\alpha \alpha} \psi)(x_1,\dots,x_{n})   \\ \nonumber
	& =[\alpha^{n-1}(x_1),\psi(x_2,\ldots,x_{n})]\\ \nonumber 
 &+\sum^{n}_{i=2}(-1)^{i}[\psi(x_1,\ldots,\hat{x_i},\ldots,x_{n}),\alpha^{n-1}(x_i)]                                    \\
	                                 & +\sum_{1\leq i<j\leq n}(-1)^{j+1}\psi(\alpha(x_1),\ldots,\alpha(x_{i-1}),[x_i,x_j],\alpha(x_{i+1}),\ldots,\widehat{\alpha(x_j)},\ldots,\alpha(x_{n})).   \nonumber                  
\end{align}
 \begin{align}
	(\partial_{\gamma \alpha} \phi) (x_1,\ldots, x_{n})    = \alpha( \phi(x_1,\ldots, x_{n}) ) - \phi(\alpha(x_1),\ldots,\alpha( x_{n})) 
\end{align}
\begin{align}
& (\partial_{\alpha \gamma} \psi)(x_1,\ldots, x_{n+1})\\ \nonumber
&= \sum^{n+1} _{i=2} (-1)^i [ [\alpha^{n-2}(x_1), \alpha^{n-2} (x_i)],\psi(x_2,\ldots, \hat{x_i}, \dots,x_{n+1})] \\ \nonumber
&+\sum_{2\leq  i<j \leq n+1}(-1)^{j}[\psi(x_1,\ldots, x_{i-1},\hat{x_i}, x_{i+1},\ldots,\hat{x_j},\ldots,x_{n+1}), [\alpha^{n-2}(x_i), \alpha^{n-2} (x_j)]]\nonumber
\end{align}

We set 
\begin{align}
\partial (\phi + \psi) & = (\partial_{\gamma \gamma} + \partial_{\gamma \alpha})(\phi) - (\partial_{\alpha \alpha} + \partial_{\alpha \gamma})(\psi)\\
\partial (\phi, \psi) & = (\partial_{\gamma \gamma} \phi - \partial_{\alpha \gamma}\psi, \partial_{\gamma \alpha}\phi - \partial_{\alpha \alpha}\psi). \nonumber
\end{align}

\begin{thm}
$\widetilde{CL}^\ast(L, L)$ together with the map $\partial (\phi , \psi) = (\partial_{\gamma \gamma} \phi - \partial_{\alpha \phi}\psi, \partial_{\gamma \alpha}\phi - \partial_{\alpha \alpha}\psi)$ is a cochain complex.
\end{thm}
\begin{proof}
We need to show that $\partial^2=0$. This is same as the following equations:
\begin{align*}
& \partial_{\gamma \gamma} \partial_{\gamma \gamma} + \partial_{\gamma \gamma} \partial_{\alpha \gamma} - \partial_{\alpha \gamma} \partial_{\alpha \alpha} - \partial_{\alpha \gamma} \partial_{\gamma \alpha} = 0,\\
& -\partial_{\alpha \alpha} \partial_{\alpha \alpha} + \partial_{\gamma \alpha} \partial_{\alpha \gamma} - \partial_{\alpha \alpha} \partial_{\gamma \alpha} + \partial_{\gamma \alpha} \partial_{\gamma \gamma} = 0.
\end{align*}
For this we verify the following equations
\begin{align*}
& \partial_{\gamma \gamma} \partial_{\gamma \gamma}= \partial_{\alpha \gamma} \partial_{\gamma \alpha},\\
& \partial_{\gamma \gamma} \partial_{\alpha \gamma} = \partial_{\alpha \gamma} \partial_{\alpha \alpha},\\
& \partial_{\alpha \alpha} \partial_{\alpha \alpha} = \partial_{\gamma \alpha} \partial_{\alpha \gamma},\\
& \partial_{\gamma \alpha} \partial_{\gamma \gamma} = \partial_{\alpha \alpha} \partial_{\gamma \alpha}.
\end{align*}
We only verify above equations for $n=1$. The proof of the general case is lengthy and can be obtained following \cite{HM19}, we omit the detail computation.
Observe that
\begin{align*}
&\partial_{\gamma \gamma} \partial_{\gamma \gamma} \phi(x_1, x_2, x_3) \\
& = [\alpha(x_1), \partial_{\gamma \gamma} \phi(x_2, x_3)] + [\partial_{\gamma \gamma}\phi (x_1, x_3), \alpha(x_2)] - [\partial_{\gamma \gamma}\phi (x_1, x_2), \alpha(x_3)] \\
& - \partial_{\gamma \gamma} \phi ([x_1, x_2], \alpha(x_3)) + \partial_{\gamma \gamma} \phi ([x_1, x_3], \alpha(x_2)) + \partial_{\gamma \gamma} \phi (\alpha(x_1), [x_2, x_3])\\
& = [\alpha (x_1), [x_2 ,\phi(x_3)]] + [\alpha(x_1), [\phi(x_2), x_3]] - [\alpha(x_1), \phi([x_2,x_3])] + [[x_1, \phi(x_3)], \alpha(x_2)] \\
& + [[\phi(x_1), x_3], \alpha(x_2)] - [\phi([x_1, x_3]), \alpha(x_2)] - [[x_1, \phi(x_2)], \alpha(x_3)] - [[\phi(x_1), x_2], \alpha(x_3)]\\
& + [\phi[x_1, x_2], \alpha(x_3)] - [[x_1, x_2], \phi \alpha(x_3)] - [\phi[x_1, x_2], \alpha(x_3)] + \phi([[x_1, x_2], \alpha(x_3)])\\
&+ [[x_1, x_3], \phi\alpha(x_2)] + [\phi[x_1, x_3], \alpha(x_2)] - \phi([[x_1,x_3], \alpha(x_2)]) + [\alpha(x_1), \phi[x_2, x_3]] \\
& + [\phi \alpha(x_1), [x_2, x_3]] - \phi([\alpha(x_1), [x_2, x_3]])\\
& = [[x_1, x_2], \alpha \phi (x_3)] - [[x_1,x_2], \phi \alpha(x_3)] - [[x_1,x_3], \alpha \phi(x_2)] + [[x_1, x_3], \phi \alpha(x_2)] \\
& - [\alpha \phi(x_1), [x_2, x_3]] + [\phi \alpha(x_1), [x_2, x_3]].
\end{align*}
Note that in the above computation the third equality is obtained by using the Hom-Leibniz identity and cancellation of terms.
On the other hand, we have
\begin{align*}
&\partial_{\alpha \gamma} \partial_{\gamma \alpha} \phi (x_1, x_2, x_3)\\
& = [[x_1, x_2], \partial_{\gamma \alpha} \phi (x_3)] - [[x_1, x_3], \partial_{\gamma \alpha} \phi (x_2)] - [\partial_{\gamma \alpha} \phi(x_1), [x_2, x_3]] \\
& = [[x_1, x_2], \alpha \phi (x_3)] - [[x_1,x_2], \phi \alpha(x_3)] - [[x_1,x_3], \alpha \phi(x_2)] + [[x_1, x_3], \phi \alpha(x_2)] \\
& - [\alpha \phi(x_1), [x_2, x_3]] + [\phi \alpha(x_1), [x_2, x_3]].
\end{align*}
Thus, $\partial_{\gamma \gamma} \partial_{\gamma \gamma} = \partial_{\alpha \gamma} \partial_{\gamma \alpha}$.
\begin{align*}
& \partial_{\gamma \alpha} \partial_{\gamma \gamma} \phi (x_1, x_2) \\
& = \alpha (\partial_{\gamma \gamma} \phi(x_1, x_2))- \partial_{\gamma \gamma} \phi (\alpha(x_1), \alpha(x_2)) \\
& = [\alpha(x_1), \alpha \phi (x_2)] + [\alpha \phi (x_1), \alpha(x_2)] - \alpha \phi ([x_1, x_2]) - [\alpha(x_1), \phi \alpha (x_2)] \\
& - [\phi \alpha (x_1), \alpha(x_2)] + \phi \alpha ([x_1, x_2]).
\end{align*}
On the other hand, we have
\begin{align*}
&\partial_{\alpha \alpha} \partial_{\gamma \alpha} \phi (x_1, x_2)\\
& = [\alpha(x_1), \partial_{\gamma \alpha} \phi (x_2)] + [\partial_{\gamma \alpha} \phi (x_1),\alpha(x_2)] - \partial_{\gamma \alpha} \phi ([x_1, x_2])\\
& = [\alpha(x_1), \alpha \phi (x_2)] - [\alpha(x_1), \phi \alpha(x_2)] + [\alpha \phi (x_1), \alpha(x_2)] - [\phi(\alpha(x_1)), \alpha(x_2)] \\
& - \alpha \phi ([x_1, x_2]) + \phi ([\alpha(x_1), \alpha(x_2)]).
\end{align*}
Thus, $\partial_{\gamma \alpha} \partial_{\gamma \gamma} = \partial_{\alpha \alpha} \partial_{\gamma \alpha}$.
As $\widetilde{CL}^1_{\alpha} (L, L) = \lbrace 0 \rbrace$, we have
\begin{align*}
& \partial_{\gamma \gamma} \partial_{\alpha \gamma}\psi (x_1, x_2) = \partial_{\alpha \gamma} \partial_{\alpha \alpha}\psi (x_1, x_2) = 0 ,\\
& \partial_{\alpha \alpha} \partial_{\alpha \alpha}\psi (x_1, x_2) = \partial_{\gamma \alpha} \partial_{\alpha \gamma}\psi (x_1, x_2) =  0.
\end{align*}
Therefore, for $n=1$ we proved $\partial^2 =0$.
\end{proof}

We denote the cohomology of the cochain complex $\big( \widetilde{CL}^\ast(L, L), \partial \big)$ by $\widetilde{HL}^\ast (L, L)$ and call it $\alpha$-type Hom-Leibniz cohomology of $L$ with coefficients in itself.

\begin{remark}
Note that $\alpha$-type cohomology for multiplicative Hom-Leibniz algebras generalizes the cohomology introduced in \cite{CA}. To show this we consider only those elements in $\widetilde{CL}^n(L, L)$ where second summand is zero, that is, $\widetilde{CL}_\alpha^n(L, L) = \lbrace 0 \rbrace.$  Thus, we have elements of the form $(\phi, 0)$. We define a subcomplex of $\widetilde{CL}^n(L, L)$ as follows:
\begin{align*}
{CL}_\alpha^n(L, L) & = \lbrace (\phi, 0) \in \widetilde{CL}^n(L, L) \mid  \partial_{\gamma \alpha} \phi = 0 \rbrace\\
                                & = \lbrace \phi \in \widetilde{CL}_\gamma^n(L, L) \mid \alpha \circ \phi = \phi \circ \alpha^{\otimes n}\rbrace.
\end{align*}
The map $\partial_{\gamma \gamma}$ defines a diffential on this complex and this complex is same as the complex defined in \cite{CA}. Thus, $\alpha$-type cohomology generalizes the cohomology developed in \cite{CA}.
\end{remark}
\section{Deformation theory of Hom-Leibniz algebra structure}\label{sec 4}
In this section, we introduce one-parameter formal deformation theory for multiplicative Hom-Leibniz algebras and discuss how $\alpha$-type cohomology controls deformations. We only consider multiplicative Hom-Leibniz algebras.
\begin{defn}
A one-parameter formal deformation of multiplicative Hom-Leibniz algebra $(L,[.,.],\alpha)$ is given by a $\mathbb{K}[[t]]$-bilinear map $m_t:L[[t]]\times L[[t]]\to L[[t]]$ and a $\mathbb{K}[[t]]$-linear map $\alpha_t:L[[t]]\to L[[t]]$ of the forms
$$m_t=\sum_{i\geq 0} m_it^i\,\text{and}\, \alpha_t=\sum_{i\geq 0}\alpha_it^i,$$
such that,
\begin{enumerate}
\item For all $i\geq 0$, $m_i:L\times L\to L$ is a $\mathbb{K}$-bilinear map, and $\alpha_i:L\to L$ is a $\mathbb{K}$-linear map.
\item $m_0(x,y)=[x,y]$ is the bracket and $\alpha_0=\alpha$ is the structure map of $L$.
\item $(L[[t]],m_t,\alpha_t)$ satisfies the Hom-Leibniz identity, that is, \\ $m_t(\alpha_t(x),m_t(y,z))=m_t(m_t(x,y),\alpha_t(z))-m_t(m_t(x,z),\alpha_t(y)).$\label{deform equ}
\item The map $\alpha_t$ is multiplicative, that is, $m_t (\alpha_t(x), \alpha_t(y)) = \alpha_t (m_t(x, y))$.\label{deform equ mult}
\end{enumerate}
\end{defn}

 Condition  (\ref{deform equ}) in the last definition is equivalent to
 \begin{align}
\label{deform equ 1} \sum_{\substack{i+j+k=n\\i,j,k\geq 0}}m_i(\alpha_j(x),m_k(y,z))-m_i(m_k(x,y),\alpha_j(z))+m_i(m_k(x,z),\alpha_j(y))=0.
 \end{align}
 
 Condition  (\ref{deform equ mult}) in the last definition is equivalent to
  \begin{align}
\label{deform equ 12} \sum_{\substack{i+j+k=n\\i,j,k\geq 0}}m_i(\alpha_j(x), \alpha_k(y))- \sum_{\substack{i+j=n\\i,j\geq 0}}\alpha_i(m_j(x,y))=0.
 \end{align}
 
 For a Hom-Leibniz algebra $(L,[.,.],\alpha)$, a $\alpha_j$-associator is a map,
 \begin{align*}
 Hom(L^{\times 2},L)\times &Hom(L^{\times 2},L)\to Hom(L^{\times 3},L),\\
 &(m_i,m_k)\mapsto m_i\circ_{\alpha_j}m_k,
 \end{align*}
 defined as
 
 $m_i\circ_{\alpha_j}m_k(x,y,z)=m_i(\alpha_j(x),m_k(y,z))-m_i(m_k(x,y),\alpha_j(z))+m_i(m_k(x,z),\alpha_j(y))$.
 By using $\alpha_j$-associator, the deformation equation may be written as 
 \begin{align*}
 &\sum_{i,j,k\geq 0} (m_i\circ_{\alpha_j}m_k)t^{i+j+k}=0,\\
&\sum_{n\geq 0}\bigg( \sum_{\substack{i + j + k= n\\ i, j, k \geq 0}} (m_i\circ_{\alpha_j}m_k) \bigg)t^n=0.  
 \end{align*}
 Thus, for $n=0,1,2,\ldots$,  we have the following infinite equations:
\begin{align}
  \label{deform equ 2}\sum_{\substack{i + j + k= n\\ i, j, k \geq 0}} (m_i\circ_{\alpha_j}m_k)=0.
 \end{align}
 We can rewrite the Equation (\ref{deform equ 2}) as follows:
\begin{align} \label{deform equ 3}
( \partial_{\gamma \gamma} m_n  -  \partial_{\alpha \gamma} \alpha_n ) (x, y, z) = - \sum_{\substack{i + j + k= n\\ i, j, k > 0}} (m_i\circ_{\alpha_j}m_k)(x, y, z),
\end{align} 
 where
 \begin{align}
  \partial_{\gamma \gamma} m_n(x,y,z)& = [\alpha(x), m_n(y,z)] + [m_n(x,z), \alpha(y)] - [m_n(x, y), \alpha(z)] \\
&  - m_n ([x, y], \alpha(z)) + m_n ([x,z], \alpha(y)) + m_n (\alpha(x), [y,z]), \nonumber
\end{align}  
and 
\begin{align}
\partial_{\alpha \gamma} \alpha_n (x,y,z) =  - [\alpha_n(x), [y,z]] + [[x,y], \alpha_n(z)] - [[x,z], \alpha_n(y)].
\end{align}

 From the multiplicativity of $\alpha_t$, we have
 \begin{align} \label{df alpha 1}
 \sum_{\substack{i + j + k= n\\ i, j, k \geq 0}} m_i (\alpha_j (x), \alpha_k(y)) - \sum_{\substack{i + j = n\\ i, j \geq 0}} \alpha_i (m_j (x, y)) = 0.
 \end{align}
 We can rewrite the Equation (\ref{df alpha 1}) as follows:
 \begin{align}
 (\partial_{\alpha \alpha} \alpha_n - \partial_{\gamma \alpha} m_n) (x, y) = - \sum_{\substack{i + j + k= n\\ i, j, k > 0}} m_i (\alpha_j (x), \alpha_k(y)) + \sum_{\substack{i + j = n\\ i, j > 0}} \alpha_i (m_j (x, y)) ,
 \end{align}
 where
 \begin{align}
& \partial_{\alpha \alpha} \alpha_n (x, y) = [\alpha(x), \alpha_n(y)] + [\alpha_n(x), \alpha(y)] - \alpha_n ([x, y]),\\
 & \partial_{\gamma \alpha} m_n(x, y) = \alpha m_n(x, y) - m_n (\alpha(x), \alpha(y)).
 \end{align}
 
 For $n=0$,
 \begin{align*}
&m_0\circ_{\alpha_0}m_0=0,\\
&m_0(\alpha_0(x),m_0(y,z))-m_0(m_0(x,y),\alpha_0(z))+m_0(m_0(x,z),\alpha_0(y))=0,\\
&[\alpha(x),[y,z]]-[[x,y],\alpha(z)]+[[x,z],\alpha(y)]=0.
  \end{align*}
  This is the original Hom-Leibniz relation and from the Equation (\ref{df alpha 1}) we have
  \begin{align*}
  & m_0(\alpha_0(x), \alpha_0(y)) -\alpha_0 (m_0(x, y)) =0, \\
  & \alpha[x, y] = [\alpha(x), \alpha(y)].
  \end{align*}
  This just shows $\alpha$ is multiplicative.
  
  For $n=1$, from the Equation (\ref{deform equ 2}) we have
  $$m_0\circ_{\alpha_0}m_1+m_1\circ_{\alpha_0}m_0+m_0\circ_{\alpha_1}m_0=0,$$
  $[\alpha(x),m_1(y,z)]-[m_1(x,y),\alpha(z)]+[m_1(x,z),\alpha(y)]+m_1(\alpha(x),[y,z])-m_1([x,y],\alpha(z))+m_1([x,z],\alpha(y)) + [\alpha_1(x), [y, z]] - [[x,y], \alpha_1(z)] + [[x,z], \alpha_1(y)]=0$. 
  
This is same as
$$\partial_{\gamma \gamma} m_1(x,y,z) - \partial_{\alpha \gamma} \alpha_1 (x, y, z)=0.$$
Now from the multiplicative part of the deformation, we have
$$[\alpha(x), \alpha_1(y)] + [\alpha_1(x), \alpha(y)] + m_1(\alpha(x), \alpha(y)) - \alpha(m_1(x,y)) - \alpha_1 [x, y] = 0.$$
This is same as 
$$\partial_{\alpha \alpha} \alpha_1(x, y) - \partial_{\gamma \alpha} m_1 (x, y)= 0.$$
Thus, we have
$$\partial (m_1, \alpha_1) = 0.$$
\begin{defn}
The infinitesimal of the deformation $(m_t, \alpha_t)$ is the pair $(m_1, \alpha_1)$. Suppose more  generally that $(m_n, \alpha_n)$ is the first non-zero term of $(m_t, \alpha_t)$ after $(m_0, \alpha_0)$, such $(m_n, \alpha_n)$ is called a $n$-infinitesimal of the deformation.
\end{defn}
Therefore, we have the following theorem.
\begin{thm}\label{cocycle}
Let $(L,[.,.],\alpha)$ be a Hom-Leibniz algebra, and $(L_t,m_t,\alpha_t)$ be its one-parameter deformation then the infinitesimal of the deformation is a $2$-cocycle of the $\alpha$-type Hom-Leibniz cohomology.
\end{thm}

\subsection{Obstructions of deformations}
Now we discuss obstructions of deformations for multiplicative Hom-Leibniz algebras from the cohomological point of view.
\begin{defn}
A $n$-deformation of a Hom-Leibniz algebra is a formal deformation of the forms
$$m_t=\sum^n_{i=0}m_it^i,~~~ \alpha_t=\sum^n_{i=0}\alpha_it^i,$$
\end{defn}
such that $m_t$ and $\alpha_t$ satisfies the Hom-Leibniz identity,
$$m_t(\alpha_t(x),m_t(y,z))=m_t(m_t(x,y),\alpha_t(z))-m_t(m_t(x,z),\alpha_t(y)),$$\label{n-deform}
and $\alpha_t$ is multiplicative,
$$m_t (\alpha_t(x), \alpha_t(y)) = \alpha_t (m_t(x, y)).$$

We say a $n$-deformation $(m_t, \alpha_t)$ of a Hom-Leibniz algebra is extendable to a $(n+1)$-deformation if there is an element $m_{n+1}\in \widetilde{CL}_\gamma^2(L, L)$ and $\alpha_{n+1} \in \widetilde{CL}_\alpha^2(L, L)$ such that
\begin{align*}
&\bar{m_t}=m_t+m_{n+1}t^{n+1},\\
&\bar{\alpha_t}=\alpha_t+\alpha_{n+1}t^{n+1},
\end{align*}
and $(\bar{m_t}, \bar{\alpha_t})$ satisfies all the conditions of one-parameter formal deformations.

The $(n+1)$-deformation $(\bar{m_t}, \bar{\alpha_t})$ gives us the following equations.
\begin{align}
\label{obs deform equ 1} \sum_{\substack{i+j+k=n+1\\i,j,k\geq 0}}m_i(\alpha_j(x),m_k(y,z))-m_i(m_k(x,y),\alpha_j(z))+m_i(m_k(x,z),\alpha_j(y))=0.
 \end{align}
 \begin{align}
\label{obs deform equ 12} \sum_{\substack{i+j+k=n+1\\i,j,k\geq 0}}m_i(\alpha_j(x), \alpha_k(y))- \sum_{\substack{i+j=n+1\\i,j\geq 0}}\alpha_i(m_j(x,y))=0.
 \end{align}
 This is same as the following equations
 \begin{align*}
&( \partial_{\gamma \gamma} m_{n+1}  -  \partial_{\alpha \gamma} \alpha_{n+1} ) (x, y, z) \\
&=-\sum_{\substack{i+j+k=n+1\\i,j,k>0}}m_i(\alpha_j(x),m_k(y,z))-m_i(m_k(x,y),\alpha_j(z))+m_i(m_k(x,z),\alpha_j(y))\\
&=-\sum_{\substack{i+j+k=n+1\\i,j,k>0}}m_i\circ_{\alpha_j}m_k.
\end{align*}
\begin{align*}
 (\partial_{\alpha \alpha} \alpha_{n+1} - \partial_{\gamma \alpha} m_{n+1}) (x, y) = - \sum_{\substack{i + j + k= n+1\\ i, j, k > 0}} m_i (\alpha_j (x), \alpha_k(y)) + \sum_{\substack{i + j = n+1\\ i, j > 0}} \alpha_i (m_j (x, y)).
\end{align*}
We define the $n$th obstruction to extend a deformation of Hom-Leibniz algebra of order $n$ to order $n+1$  as  $\text{Obs}^n = (\text{Obs}^n_\gamma, \text{Obs}^n_\alpha)$, where
\begin{align}
\label{obs equ 222}&\text{Obs}^n_\gamma (x, y, z):=\sum_{\substack{i+j+k=n+1\\i,j,k>0}}(m_i\circ_{\alpha_j}m_k) (x, y, z)= (\partial_{\gamma \gamma} m_{n+1}  -  \partial_{\alpha \gamma} \alpha_{n+1})(x, y, z),\\
&\text{Obs}^n_\alpha (x, y)\\ \nonumber
&:=- \sum_{\substack{i + j + k= n+1\\ i, j, k > 0}} m_i (\alpha_j (x), \alpha_k(y)) + \sum_{\substack{i + j = n+1\\ i, j > 0}} \alpha_i (m_j (x, y)) \\ \nonumber
&= (\partial_{\alpha \alpha} \alpha_{n+1} - \partial_{\gamma \alpha} m_{n+1}) (x, y).
\end{align}
Thus,  $(\text{Obs}^n_\gamma, \text{Obs}^n_\alpha) \in \widetilde{CL}^3(L, L)$ and $(\text{Obs}^n_\gamma, \text{Obs}^n_\alpha) = \partial (m_{n+1}, \alpha_{n+1})$.
\begin{theorem}
A deformation of order $n$ extends to a deformation of order $n+1$ if and only if cohomology class of  $\text{Obs}^n$ vanishes.
\end{theorem}
\begin{proof}
Suppose a deformation $(m_t, \alpha_t)$ of order $n$ extends to a deformation of order $n+1$. From the obstruction equations, we have
$$\text{Obs}^n = (\text{Obs}^n_\gamma, \text{Obs}^n_\alpha) = \partial (m_{n+1}, \alpha_{n+1}).$$
As $\partial \circ \partial=0$, we get the cohomology class of  $\text{Obs}^n$ vanishes.

Conversely, suppose the cohomology class of  $\text{Obs}^n$ vanishes, that is,
$$\text{Obs}^n=\partial (m_{n+1}, \alpha_{n+1}),$$
for some $2$-cochains $(m_{n+1}, \alpha_{n+1})$. We define $(m'_t, \alpha'_t)$ extending the deformation $(m_t, \alpha_t)$ of order $n$ as follows-
\begin{align*}
&m'_t=m_t+m_{n+1}t^{n+1},\\
&\alpha'_t=\alpha_t+ \alpha_{n+1}t^{n+1}.
\end{align*}
The map $m'_t, \alpha'_t$ satisfy the following equations for all $x,y,z\in L$.
\begin{align*}
 &\sum_{\substack{i+j+k=n+1\\i,j,k\geq 0}}m_i(\alpha_j(x),m_k(y,z))-m_i(m_k(x,y),\alpha_j(z))+m_i(m_k(x,z),\alpha_j(y))=0,\\
 &\sum_{\substack{i+j+k=n+1\\i,j,k\geq 0}}m_i(\alpha_j(x), \alpha_k(y))- \sum_{\substack{i+j=n+1\\i,j\geq 0}}\alpha_i(m_j(x,y))=0.
 \end{align*}
 Thus, $(m'_t, \alpha'_t)$ is a deformation of order $n+1$ which extends the deformation $(m_t, \alpha_t)$ of order $n$.
\end{proof}
\begin{cor}
If $\widetilde{HL}^3 (L, L)=0$ then any $2$-cocycle gives a one-parameter formal deformation of $(L,[.,.],\alpha)$.
\end{cor}

\subsection{Equivalent and trivial deformations}
 Suppose $L_t=(L,m_t,\alpha_t)$ and $L\rq_t=(L,m\rq_t,\alpha\rq_t)$ be two one-parameter Hom-Leibniz algebra deformations of $(L,[.,.],\alpha)$, where $m_t=\sum_{i\geq 0} m_it^i,\, \alpha_t=\sum_{i\geq 0}\alpha_it^i$ and $m\rq_t=\sum_{i\geq 0} m\rq_it^i,\, \alpha\rq_t=\sum_{i\geq 0}\alpha\rq_it^i$. 
\begin{defn}
 Two deformations $L_t$ and $L\rq_t$ are said to be equivalent if there exists a $\mathbb{K}[[t]]$-linear isomorphism $\Psi_t:L[[t]]\to L[[t]]$ of the form $\Psi_t=\sum_{i\geq 0}\psi_it^i$, where $\psi_0=Id$ and $\psi_i:L\to L$ are $\mathbb{K}$-linear maps such that the following relations holds:
 \begin{align}\label{equivalent}
 \Psi_t\circ m_t\rq=m_t\circ (\Psi_t\otimes \Psi_t)\,\,\,\text{and}\,\,\, \alpha_t\circ \Psi_t=\Psi_t\circ \alpha\rq_t.
 \end{align}
 \end{defn}
 \begin{defn}
 A deformation $L_t$ of a Hom-Leibniz algebra $L$ is called trivial if $L_t$ is equivalent to $L$. A Hom-Leibniz algebra $L$ is called rigid if it has only trivial deformation upto equivalence.
 \end{defn}
 Condition (\ref{equivalent}) may be written as
 \label{equivalent 11}$$\Psi_t(m\rq_t(x,y))=m_t(\Psi_t(x),\Psi_t(y))\,\,\, \text{and}\,\,\, \alpha_t(\Psi_t(x))=\Psi_t(\alpha\rq_t(x)),\,\,\,\text{for all}~ x,y\in L.$$
 The above conditions is equivalent to the following equations:
 \begin{align}
 &\sum_{i\geq 0}\psi_i\bigg( \sum_{j\geq 0}m\rq_j(x,y)t^j \bigg)t^i=\sum_{i\geq 0}m_i\bigg( \sum_{j\geq 0}\psi_j(x)t^j,\sum_{k\geq 0}\psi_k(y)t^k \bigg)t^i,\\
 &\sum_{i\geq 0}\alpha_i\bigg( \sum_{j\geq 0}\psi_j(x)t^j \bigg)t^i=\sum_{i\geq 0}\psi_i\bigg( \sum_{j\geq 0}\alpha\rq_j(x)t^j \bigg)t^i.
 \end{align}
 This is same as the following equations:
 \begin{align}
 \label{equivalent 10}&\sum_{i,j\geq 0}\psi_i(m\rq_j(x,y))t^{i+j}=\sum_{i,j,k\geq 0}m_i(\psi_j(x),\psi_k(y))t^{i+j+k},\\
 &\label{equivalent 101}\sum_{i,j \geq 0}\alpha_i(\psi_j(x))t^{i+j}=\sum_{i,j\geq 0}\psi_i(\alpha\rq_j(x))t^{i+j}.
 \end{align}
 Comparing constant terms on both sides of the above equations, we have
 \begin{align*}
 &m\rq_0(x,y)=m_0(x,y)=[x,y],\,\,\,\text{as}\,\,\,\psi_0=Id,\\
 &\alpha_0(x)=\alpha\rq_0(x)=\alpha(x).
  \end{align*}
  Now comparing coefficients of $t$, we have
  \begin{align}\label{equivalent main}
&m\rq_1(x,y)+\psi_1(m\rq_0(x,y))=m_1(x,y)+m_0(\psi_1(x),y)+m_0(x,\psi(y)),\\
\label{equivalent main 1}&\alpha_1(x)+\alpha_0(\psi_1(x))=\alpha\rq_1(x)+\psi_1(\alpha\rq_0(x)).
  \end{align}
  The Equations (\ref{equivalent main}) and (\ref{equivalent main 1}) are same as
  \begin{align*}
& m\rq_1(x,y)-m_1(x,y)=[\psi_1(x),y]+[x,\psi_1(y)]-\psi_1([x,y])=\partial_{\gamma \gamma} \psi_1(x,y).\\
& \alpha_1\rq (x) - \alpha_1(x) = \alpha(\psi_1(x)) - \psi_1(\alpha(x)) = \partial_{\gamma \alpha} \psi_1(x).
\end{align*}
Thus, we have the following proposition.
  \begin{prop}
 Two equivalent deformations have cohomologous infinitesimals.
  \end{prop}
  \begin{proof}
  Suppose $L_t=(L,m_t,\alpha_t)$ and $L\rq_t=(L,m\rq_t,\alpha\rq_t)$ be two equivalent one-parameter Hom-Leibniz deformations of $(L,[.,.],\alpha)$. Suppose $(m_n, \alpha_n)$ and $(m\rq_n, \alpha\rq_n)$ be two $n$-infinitesimals of the deformations $(m_t, \alpha_t)$ and $(m\rq_t, \alpha\rq_t)$ respectively. Using Equation (\ref{equivalent 10}) we get,
  \begin{align*}
  &m\rq_n(x,y)+\psi_n(m\rq_0(x,y))=m_n(x,y)+m_0(\psi_n(x),y)+m_0(x,\psi_n(y)),\\
  &m\rq_n(x,y)-m_n(x,y)=m_0(\psi_n(x),y)+m_0(x,\psi_n(y))-\psi_n(m\rq_0(x,y)),\\
 &m\rq_n(x,y)-m_n(x,y)=[\psi_n(x),y]+[x,\psi_n(y)]-\psi_n([x,y])=\partial_{\gamma \gamma} \psi_n(x,y). 
  \end{align*}
  Using equation (\ref{equivalent 101}) we get,
  \begin{align*}
&\alpha_0(\psi_n(x)) - \psi_0(\alpha\rq_n(x)) + \alpha_n(\psi_0(x)) - \psi_n(\alpha\rq_0(x))=0,\\
&\alpha_n(x) - \alpha\rq_n(x) = \psi_n(\alpha(x)) - \alpha(\psi_n(x)) = \partial_{\gamma \alpha}\psi_n(x).
  \end{align*}
  Thus, infinitesimals of two deformations determines same cohomology class.
    \end{proof}
  \begin{theorem}
  A non-trivial deformation of a Hom-Leibniz algebra is equivalent to a deformation whose infinitesimal is not a coboundary.
  \end{theorem}
  \begin{proof}
 Let $(m_t,\alpha_t)$ be a deformation of Hom-Leibniz algebra $L$ and $(m_n, \alpha_n)$ be the $n$-infinitesimal of the deformation for some $n\geq 1$. Then by Theorem (\ref{cocycle}), $(m_n, \alpha_n)$ is a $2$-cocycle, that is, $\partial (m_n, \alpha_n)=0$. Suppose $m_n=-\partial_{\gamma \gamma}\phi_n$ and $\alpha_n = -\partial_{\gamma \alpha}\phi_n$for some $\phi_n\in \widetilde{CL}_\gamma^1(L, L)$, that is, $(m_n, \alpha_n)$ is a coboundary. We define a formal isomorphism $\Psi_t$ of $L[[t]]$ as follows:
  $$\Psi_t(a)=a+\phi_n(a)t^n.$$
  We set
  $$\bar{m_t}=\Psi^{-1}_t\circ m_t\circ (\Psi_t\otimes\Psi_t) \,\,\,\text{and}\,\,\,\bar{\alpha_t}=\Psi^{-1}_t\circ\alpha_t\circ\Psi_t$$
  Thus, we have a new deformation $\bar{L_t}$ which is isomorphic to $L_t$. By expanding the above equations and comparing coefficients of $t^n$, we get
  $$\bar{m_n}-m_n=\partial_{\gamma \gamma} \phi_n,~~~\bar{\alpha_n} - \alpha = \partial_{\gamma \alpha}\phi_n.$$
  Hence, $\bar{m_n}=0, \bar{\alpha_n}=0$. By repeating this argument, we can kill off any infinitesimal which is a coboundary. Thus, the process must be stopped if the deformation is non-trivial. 
  \end{proof}
 \begin{cor}
 Let $(L,[.,.],\alpha)$ be a Hom-Leibniz algebra. If $\widetilde{HL}^2 (L, L)=0$ then $L$ is rigid.
 \end{cor}
\section{Group action and equivariant cohomology}\label{sec 5}
The notion of a finite group action on a Leibniz algebra was introduced by authors in \cite{MS19}. In this section, we introduce a notion of a finite group action on Hom-Leibniz algebra. We also define an equivariant cohomology of Hom-Leibniz algebra equipped with action of a finite group.
\begin{defn}
Let $G$ be a finite group and $(L,[.,.],\alpha)$ be a Hom-Leibniz algebra. We say group $G$ acts on the Hom-Leibniz algebra $L$ from the left if there is a funtion
$$\Phi:G\times L\to L,$$
satisfying
\begin{enumerate}
\item For each $g\in G$, the map $\psi_g:L\to L,\,x\mapsto gx$ is a $\mathbb{K}$-linear map.
\item $ex=x$, where $e$ denotes identity element of the group $G$.
\item For all $g_1,g_2\in G$ and $x\in L$, $(g_1g_2)x=g_1(g_2x)$.
\item For all $g\in G$ and $x,y\in L$, $g[x,y]=[gx,gy]$ and $\alpha(gx)=g\alpha(x).$
\end{enumerate}
\end{defn}
We may write a Hom-Leibniz algebra $(L,[.,.],\alpha)$  equipped with a finite group action $G$ as $(G,L,[.,.],\alpha)$. 
An alternative way to present the above definition is the following:
\begin{prop}
Let $G$ be a finite group and $(L,[.,.],\alpha)$ be a Hom-Leibniz algebra. The group $G$ acts on $L$ from the left if and only if there is a group homomorphism
\begin{align*}
\Psi:G\to Iso_{Hom\text{-}Leib}(L),\,g\mapsto\psi_g,
\end{align*}
where $Iso_{Hom\text{-}Leib}(L)$ denotes group of ismorphisms of Hom-Leibniz algebras from $L$ to $L$.
\end{prop}

Let $M,M\rq$ be Hom-Leibniz algebras equipped with actions of group $G$. We say a $\mathbb{K}$-linear map $f:M\to M\rq$ is equivariant  if for all $g\in G$ and $x\in M$, $f(gx)=gf(x)$. We write the set of all equivariant maps from $M$ to $M\rq$ as $\text{Hom}^G_\mathbb{K}(M,M\rq)$.

A $G$-Hom-vector space is a Hom-vector space $(M,\beta)$ together with a action of $G$ on $M$, and $\beta:M\to M$ is an equivariant map.
We denote an equivariant Hom-vector space as triple $(G,M,\beta)$.

\begin{exam}
Any $G$-Hom-vector space $(G,M,\beta)$ together with the trivial bracket (i.e. $[x,y]=0$ for all $x,y\in M$) is a Hom-Leibniz algebra equipped with an action of $G$.
\begin{exam}\label{example-3}
Let $V$ be $\mathbb K$-module which is a representation space of a finite group $G.$ On 
$$\bar{T}(V) = V\oplus V^{\otimes 2}\oplus \cdots \oplus V^{\otimes n}\oplus \cdots $$ there is a unique bracket that makes it into a Hom-Leibniz algebra by taking $\alpha=\text{Id}$ and verifies 
$$v_1 \otimes v_2 \otimes  \cdots \otimes v_n = [\cdots [[v_1,v_2],v_3],\cdots ,v_n]~\mbox{for}~v_i\in V~\mbox{and}~i=1,\ldots,n.$$ This is the free Hom-Leibniz algebra over the $\mathbb{K}$-module $V$. The linear action of $G$ on $V$ extends naturally to an action on $\bar{T}(V)$.
\end{exam}   
\end{exam}
\begin{defn}
 Let $(G,L,[.,.],\alpha)$ be a Hom-Leibniz algebra equipped with an action of a finite group $G$. A  $G$-bimodule over $L$ is a $G$-Hom-vector space  $(G,M,\beta)$ together with two $L$-actions (left and right multiplications), $m_l:L\otimes M\to M $ and $m_r:M\otimes L\to M $ such that $m_l,m_r$ satisfying the following conditions:
  \begin{align*}
  & m_l(gx, gm) = g m_l(x, m),\\
  & m_r(gm, gx) =g m_r(m,x),\\
 &\beta (m_l (x, m)) = m_l(\alpha(x), \beta(m)),\\
 &\beta (m_r (m, x)) = m_r(\beta(m), \alpha(x)),\\
 &m_r (\beta(m), [x,y]) = m_r(m_r(m, x), \alpha(y)) - m_r (m_r(m, y), \alpha(x)),\\
 & m_l (\alpha(x), m_r(m, y))=  m_r(m_l(x, m), \alpha (y))- m_l([x, y], \beta(m)),\\
 & m_l(\alpha(x), m_l(y, m)) = m_l([x, y], \beta(m)) - m_r(m_l(x, m), y),
 \end{align*}
 for any $x, y \in L, m\in M$, and $g \in G$.
 \end{defn}
 \begin{remark}
 Any Hom-Leibniz algebra equipped with an action of a finite group $G$ is a G-bimodule over itself. In this paper, we shall only consider G-bimodule over itself.
 \end{remark}
We now introduce an equivariant cohomology of Hom-Leibniz algebras  $L$ equipped with an action of a finite group $G$.

Set 
\begin{align*}
&\widetilde{CL}^n_{G}(L,L)\\
&:=\lbrace (c_\gamma, c_\alpha) \in \widetilde{CL}^n(L, L) : c_\gamma(\psi_g(x_1),\ldots,\psi_g(x_n))=gc_\gamma(x_1,\ldots,x_n),\\
&~ c_\alpha(\psi_g(x_1),\ldots,\psi_g(x_{n-1}))=gc_\alpha(x_1,\ldots,x_{n-1})\rbrace\\
&=\lbrace (c_\gamma, c_\alpha) \in \widetilde{CL}^n(L, L) : c_\gamma(gx_1,\ldots,gx_n)=gc_\gamma(x_1,\ldots,x_n),\\
&~ c_\alpha(gx_1,\ldots,gx_{n-1})=gc_\alpha(x_1,\ldots,x_{n-1}) \rbrace.
\end{align*}
Here $\widetilde{CL}^n(L, L)$ is $n$-cochain group of the Hom-Leibniz algebra $(L,[.,.],\alpha)$ and $\widetilde{CL}^n_{G}(L,L)$ consists of all $n$-cochains which are equivariant. Clearly, $\widetilde{CL}^n_{G}(L,L)$ is a submodule of $\widetilde{CL}^n(L, L)$ and $(c_\gamma, c_\alpha) \in \widetilde{CL}^n_{G}(L,L)$ is called an invariant $n$-cochain. 
\begin{lemma}
If a $n$-cochain $(c_\gamma, c_\alpha)$ is invariant then $\partial(c_\gamma, c_\alpha)$ is also an invariant $(n+1)$-cochain. In otherwords,
$$(c_\gamma, c_\alpha) \in \widetilde{CL}^n_{G}(L,L)\implies \partial(c_\gamma, c_\alpha)\in \widetilde{CL}^{n+1}_{G}(L,L).$$
\end{lemma}
\begin{proof}
As $(c_\gamma, c_\alpha) \in \widetilde{CL}^n_{G}(L,L)$, we have 
\begin{align*}
& c_\gamma(gx_1,gx_2,\ldots,gx_{n})=gc_\gamma(x_1,x_2,\ldots,x_{n}),\\
& c_\alpha(gx_1,gx_2,\ldots,gx_{n-1})=gc_\alpha(x_1,x_2,\ldots,x_{n-1}),
\end{align*}
 for all $g\in G$ and $x_1,\ldots,x_{n+1}\in L$. It is enough to show that the four differentials $\partial_{\gamma \gamma}, \partial_{\gamma \alpha}, \partial_{\alpha \alpha}, \partial_{\alpha \gamma}$ respect the group action.
Observe that 
\begin{align*}
&\partial_{\gamma \gamma}(c_\gamma)(\psi_g(x_1),\psi_g(x_2),\ldots,\psi_g(x_{n+1}))\\
&=\partial_{\gamma \gamma}(c_\gamma)(gx_1,gx_2,\ldots,gx_{n+1})\\
                                                  &=[\alpha^{n-1}(gx_1), c_\gamma(gx_2,\ldots,gx_n)]+\sum^{n+1}_{i=2}(-1)^{i}[c_\gamma(gx_1,\cdots,\hat{gx_i},\ldots,gx_{n+1}),\alpha^{n-1}g(x_i)]\\
&+\sum_{1\leq i<j\leq n+1}(-1)^{j+1} c_\gamma(\alpha(gx_1),\ldots,\alpha(gx_{i-1}),[gx_i,gx_j],\alpha(gx_{i+1}),\ldots,\widehat{\alpha(gx_j)},\ldots,\alpha(gx_{n+1})).\\
&=[g\alpha^{n-1}(x_1),g c_\gamma(x_2,\ldots,x_{n+1})]+\sum^{n+1}_{i=2}(-1)^{i}[gc_\gamma(x_1,\cdots,\hat{x_i},\ldots,x_n),g\alpha^{n-1}(x_i)]\\
&+\sum_{1\leq i<j\leq n+1}(-1)^{j+1}gc_\gamma(\alpha(x_1),\ldots,\alpha(x_{i-1}),[x_i,x_j],\alpha(x_{i+1}),\ldots,\widehat{\alpha(gx_j)},\ldots,\alpha(x_{n+1})).\\
&=g\partial_{\gamma \gamma}(c_\gamma)(x_1,x_2,\ldots,x_{n+1}).
\end{align*}
On the other hand, we have
\begin{align*}
&(\partial_{\gamma \alpha} c_\gamma) (gx_1,\ldots, gx_{n})\\
&= \alpha( c_\gamma(gx_1,\ldots, gx_{n}) ) - c_\gamma(\alpha(gx_1),\ldots,\alpha( gx_{n})) \\
                                                                                                         &= g \big(\alpha( c_\gamma(x_1,\ldots, x_{n}) ) - c_\gamma(\alpha(x_1),\ldots,\alpha( x_{n}))\big)\\
                                                                                                         &=g(\partial_{\gamma \alpha} c_\gamma) (x_1,\ldots, x_{n}).
\end{align*}
Similarly, it is easy to show that 
\begin{align*}
&(\partial_{\alpha \alpha} c_\alpha) (gx_1,\ldots, gx_{n})= g(\partial_{\alpha \alpha} c_\alpha) (x_1,\ldots, x_{n}),\\
&(\partial_{\alpha \gamma} c_\alpha) (gx_1,\ldots, gx_{n+1})= g(\partial_{\alpha \gamma} c_\alpha) (x_1,\ldots, x_{n+1}).
\end{align*}
Thus, $\partial(c_\gamma, c_\alpha)\in \widetilde{CL}^{n+1}_{G}(L,L)$.
\end{proof}

The cochain complex $\lbrace \widetilde{CL}^\ast_{G}(L,L),\partial\rbrace $ is called an equivariant cochain complex of $(G,L,[.,.],\alpha)$. We define $n$th equivariant cohomology group of $(G,L,[.,.],\alpha)$ with coefficients over itself by
$$\widetilde{HL}_G^n (L, L):=H_n(\widetilde{CL}^\ast_{G}(L,L)).$$

\section{Equivariant formal deformation of Hom-Leibniz algebra structure}\label{formal defn}\label{sec 6}
In this section, we introduce an equivariant one-parameter formal deformation theory including the deformation of the structure map of Hom-Leibniz algebra equipped with action of a finite group $G$. We show that equivariant cohomology controls such equivariant deformations.
\begin{defn}
An equivariant one-parameter formal deformation of $(G,L,[.,.],\alpha)$ is given by $\mathbb{K}[[t]]$-bilinear and a $\mathbb{K}[[t]]$-linear map $m_t:L[[t]]\times L[[t]]\to L[[t]]$ and $\alpha_t:L[[t]]\to L[[t]]$ respectively of the forms
$$m_t=\sum_{i\geq 0}m_it^i\,\,\text{and}\,\,\alpha_t=\sum_{i\geq 0}\alpha_it^i,$$
where each $m_i:L\times L\to L$ is a $\mathbb{K}$-bilinear map and each $\alpha_i:L\to L$ is a $\mathbb{K}$-linear map satisfying the followings:
\begin{enumerate}
\item $m_0(x,y)=[x,y]$ is the original Hom-Leibniz bracket on $L$ and $\alpha_0(x,y)=\alpha(x,y)$.
\item $m_t$ and $\alpha_t$ satisfies the following Hom-Leibniz algebra condition:\label{equ defor}$$
m_t(\alpha_t(x),m_t(y,z))=m_t(m_t(x,y),\alpha_t(z))-m_t(m_t(x,z),\alpha_t(y)).
$$
\item The map $\alpha_t$ is multiplicative, that is, $m_t (\alpha_t(x), \alpha_t(y)) = \alpha_t (m_t(x, y))$.\label{equ defor mult}
\item For all $g\in G$,~$x,y\in L$ and $i\geq 0$,$$m_i(gx,gy)=gm_i(x,y)\,\,\text{and}\,\,\alpha_i(gx)=g\alpha_i(x),$$
that is, $m_i\in\text{Hom}^G_\mathbb{K}(L\otimes L, L)$ and $\alpha_i\in\text{Hom}^G_\mathbb{K}(L,L).$
\end{enumerate}
\end{defn}
For all $n\geq 0$, the Condition (\ref{equ defor}) in the above Definition is equivalent to
\begin{align}
  \label{equ deform equ 11}\sum_{\substack{i+j+k=n\\i,j,k\geq 0}}m_i(\alpha_j(x),m_k(y,z))-m_i(m_k(x,y),\alpha_j(z))+m_i(m_k(x,z),\alpha_j(y))=0.
 \end{align}
 For all $n\geq 0$, the Condition (\ref{equ defor mult}) in the above Definition is equivalent to
 \begin{align}
\label{equ deform equ 12} \sum_{\substack{i+j+k=n\\i,j,k\geq 0}}m_i(\alpha_j(x), \alpha_k(y))- \sum_{\substack{i+j=n\\i,j\geq 0}}\alpha_i(m_j(x,y))=0.
 \end{align}
 \begin{defn}
An equivariant $2$-cochain $(m_1, \alpha_1)$ is called an equivariant infinitesimal of the equivariant deformation $(m_t, \alpha_t)$. Suppose more  generally that $(m_n, \alpha_n)$ is the first non-zero term of $(m_t, \alpha_t)$ after $(m_0, \alpha_0)$, such $(m_n, \alpha_n)$ is called an equivariant $n$-infinitesimal of the equivariant deformation.
\end{defn}
\begin{prop}\label{equ cocycle}
Let $G$ be a finite group and $(L,[.,.],\alpha)$ be a Hom-Leibniz algebra. Suppose $(G,L_t,m_t,\alpha_t)$ is its equivariant one-parameter deformation then the equivariant infinitesimal of an equivariant deformation is a $2$-cocycle of the equivariant Hom-Leibniz cohomology.
\end{prop}
\begin{proof}
Let $(m_n, \alpha_n)$ be an equivariant $n$-infinitesimal of an equivariant deformation $(m_t, \alpha_t)$. Thus, $m_i=\alpha_i=0$ for all $0<i<n$ and $m_i(gx,gy)=gm_i(x,y),~ \alpha_i(gx)=g\alpha_i(x)$. From the Equation (\ref{equ deform equ 11}), we have
\begin{align*}
&[\alpha(x),m_n(y,z)]-[m_n(x,y),\alpha(z)]+[m_n(x,z),\alpha(y)]+m_n(\alpha(x),[y,z])\\
&-m_n([x,y],\alpha(z))+m_n([x,z],\alpha(y))+ [\alpha_n(x), [y,z]] - [[x,y], \alpha_n(z)] + [[x,z], \alpha_n(y)]\\
&=( \partial_{\gamma \gamma} m_n  -  \partial_{\alpha \gamma} \alpha_n ) (x, y, z)=0.
\end{align*}
From the Equation (\ref{equ deform equ 12}), we have
\begin{align*}
&[\alpha(x), \alpha_n(y)] + [\alpha_n(x), \alpha(y)] + m_n(\alpha(x), \alpha(y)) - \alpha(m_n(x,y)) - \alpha_n [x, y] \\
&= (\partial_{\alpha \alpha} \alpha_1 - \partial_{\gamma \alpha} m_1)(x, y)\\
&= 0.
\end{align*}

This is same as $\partial^2 (m_n, \alpha_n)=0$. Thus, the desired result follows.
\end{proof}

An equivariant $n$-deformation of a Hom-Leibniz algebra equipped with a finite group action is a formal deformation of the forms
$$m_t=\sum^n_{i=0}m_it^i,~~~ \alpha_t=\sum^n_{i=0}\alpha_it^i,$$
such that 
\begin{enumerate}
\item For each $0\leq i\leq n$, $m_i\in \text{Hom}^G_\mathbb{K}(L\otimes L,L)$ and $\alpha_i\in\text{Hom}^G_\mathbb{K}(L,L)$, that is, each $m_i$ and $\alpha_i$ are equivariant $\mathbb{K}$-linear maps.
\item $m_t$ satisfies the Hom-Leibniz identity, that is,

$m_t(\alpha_t(x),m_t(y,z))=m_t(m_t(x,y),\alpha_t(z))-m_t(m_t(x,z),\alpha_t(y)).$\label{equ n-deform}
\item The map $\alpha_t$ is multiplicative, that is, $m_t (\alpha_t(x), \alpha_t(y)) = \alpha_t (m_t(x, y))$.
\end{enumerate}

We say an equivariant $n$-deformation $(m_t, \alpha_t)$ of a Hom-Leibniz algebra $(G,L,[.,.],\alpha)$ is extendable to an equivariant $(n+1)$-deformation if there is an element $(m_{n+1}, \alpha_{n+1}) \in \widetilde{CL}^{n+1}_{G}(L,L)$ such that
\begin{align*}
&\bar{m_t}=m_t+m_{n+1}t^{n+1},\\
&\bar{\alpha_t}=\alpha_t+\alpha_{n+1}t^{n+1},
\end{align*}
and $(\bar{m_t}, \bar{\alpha_t})$ satisfies all the conditions of formal deformations.

For $n\geq -1$, we can rewrite the Equation (\ref{equ deform equ 11}) in the following form using Hom-Leibniz cohomology
\begin{align}
\label{eq obs deform equ 1} \sum_{\substack{i+j+k=n+1\\i,j,k\geq 0}}m_i(\alpha_j(x),m_k(y,z))-m_i(m_k(x,y),\alpha_j(z))+m_i(m_k(x,z),\alpha_j(y))=0.
 \end{align}
 \begin{align}
\label{eq obs deform equ 12} \sum_{\substack{i+j+k=n+1\\i,j,k\geq 0}}m_i(\alpha_j(x), \alpha_k(y))- \sum_{\substack{i+j=n+1\\i,j\geq 0}}\alpha_i(m_j(x,y))=0.
 \end{align}
 This is same as the following equations
 \begin{align*}
&( \partial_{\gamma \gamma} m_{n+1}  -  \partial_{\alpha \gamma} \alpha_{n+1} ) (x, y, z) \\
&=-\sum_{\substack{i+j+k=n+1\\i,j,k>0}}m_i(\alpha_j(x),m_k(y,z))-m_i(m_k(x,y),\alpha_j(z))+m_i(m_k(x,z),\alpha_j(y))\\
&=-\sum_{\substack{i+j+k=n+1\\i,j,k>0}}m_i\circ_{\alpha_j}m_k (x, y, z).
\end{align*}
\begin{align*}
 (\partial_{\alpha \alpha} \alpha_{n+1} - \partial_{\gamma \alpha} m_{n+1}) (x, y) = - \sum_{\substack{i + j + k= n+1\\ i, j, k > 0}} m_i (\alpha_j (x), \alpha_k(y)) + \sum_{\substack{i + j = n+1\\ i, j > 0}} \alpha_i (m_j (x, y)).
\end{align*}
We define $n$th obstruction to extend a deformation of Hom-Leibniz algebra of order $n$ to order $n+1$  as  $\text{Obs}_G^n = (\text{Obs}^n_{G,\gamma}, \text{Obs}^n_{G,\alpha})$, where
\begin{align}
\label{eq obs equ 222}&\text{Obs}^n_{G,\gamma}(x, y, z):=\sum_{\substack{i+j+k=n+1\\i,j,k>0}}m_i\circ_{\alpha_j}m_k (x, y, z)=(\partial_{\gamma \gamma} m_{n+1}  -  \partial_{\alpha \gamma} \alpha_{n+1})(x, y, z),\\
&\text{Obs}^n_{G,\alpha}(x, y) \\ \nonumber
& :=- \sum_{\substack{i + j + k= n+1\\ i, j, k > 0}} m_i (\alpha_j (x), \alpha_k(y)) + \sum_{\substack{i + j = n+1\\ i, j > 0}} \alpha_i (m_j (x, y)) \\ \nonumber
&= (\partial_{\alpha \alpha} \alpha_{n+1} - \partial_{\gamma \alpha} m_{n+1})(x, y).
\end{align}
\begin{lemma}
Suppose $(m_t, \alpha_t)$ is an equivariant $n$-deformations, then $\text{Obs}_G^n\in \widetilde{CL}^{3}_{G}(L,L)$ is a cocycle for all $n\geq 1$.
\end{lemma}
\begin{proof}
As for all $i\geq 0$, $m_i\in \text{Hom}^G_\mathbb{K}(L\otimes L,L)$ and $\alpha_i\in \text{Hom}^G_\mathbb{K}(L,L)$. So for all $x,y\in L$, $m_i(gx,gy)=gm_i(x,y)$ and $\alpha_i(gx)=g\alpha_i(x)$.
Now,
\begin{align*}
&\text{Obs}^n_{G,\gamma}(gx,gy,gz)
=-\sum_{\substack{i+j+k=n+1\\i,j,k>0}}m_i(\alpha_j(gx),m_k(gy,gz))\\
&-m_i(m_k(gx,gy),\alpha_j(gz))+m_i(m_k(gx,gz),\alpha_j(gy))\\
&=-g\sum_{\substack{i+j+k=n+1\\i,j,k>0}}m_i(\alpha_j(x),m_k(y,z))-m_i(m_k(x,y),\alpha_j(z))+m_i(m_k(x,z),\alpha_j(y))\\
&=g\text{Obs}^n_{G,\gamma}(x,y,z).
\end{align*}
\begin{align*}
\text{Obs}^n_{G,\alpha}(gx, gy)& =- \sum_{\substack{i + j + k= n+1\\ i, j, k > 0}} m_i (\alpha_j (gx), \alpha_k(gy)) + \sum_{\substack{i + j = n+1\\ i, j > 0}} \alpha_i (m_j (gx, gy))\\
                                                     &= g\big(\sum_{\substack{i + j + k= n+1\\ i, j, k > 0}} m_i (\alpha_j (x), \alpha_k(y)) + \sum_{\substack{i + j = n+1\\ i, j > 0}} \alpha_i (m_j (x, y))\big)\\
                                                     &= g\text{Obs}^n_{G,\alpha}(x, y).
\end{align*}
Thus, $\text{Obs}_G^n\in \widetilde{CL}^{3}_{G}(L,L).$
As $\text{Obs}_G^n(x,y,z)=\partial (m_{n+1},\alpha_{n+1})(x,y,z)$, we have $\text{Obs}_G^n$ is an equivariant cocycle.
\end{proof}
We can prove the following theorem along the same line as of non-equivariant case.
\begin{theorem}
An equivariant $n$-deformation extends to an equivariant $(n+1)$-deformation if and only if cohomology class of  $\text{Obs}_G^n$ vanishes.
\end{theorem}
\begin{cor}
If $\widetilde{HL}_G^3 (L, L)=0$ then any equivariant $2$-cocycle gives an equivariant one-parameter formal deformation of $(G,L,[.,.],\alpha)$.
\end{cor}

Finally, we study rigidity conditions for equivariant deformations. Observe that an action of a finite group $G$ on Hom-Leibniz algebra $L$ induces an action on $L[[t]]$ by bilinearity.
\label{equ equivalent defn}\begin{defn}
Given two equivariant deformations $L^G_t=(G,L,m_t,\alpha_t)$ and ${L\rq}^G_t=(L,m\rq_t,\alpha\rq_t)$ of $(G, L,[.,.],\alpha)$, where $m_t=\sum_{i\geq 0} m_it^i,\, \alpha_t=\sum_{i\geq 0}\alpha_it^i$ and $m\rq_t=\sum_{i\geq 0} m\rq_it^i,\, \alpha\rq_t=\sum_{i\geq 0}\alpha\rq_it^i$.  We say $L^G_t$ and ${L\rq}_t^G$ are equivalent if there is a formal isomorphism $\Psi_t:L[[t]]\to L[[t]]$ of the following form:
$$\Psi_t(a)=\psi_0(a)+\psi_1(a)t+\psi_2(a)t^2+\cdots,$$
Such that 
\begin{enumerate}
\item $\psi_0=Id$ and for $i\geq 1$, $\psi_i:L\to L$ are equivariant $\mathbb{K}$-linear maps.
\item $\Psi_t\circ m_t\rq=m_t\circ (\Psi_t\otimes \Psi_t)\,\,\,\text{and}\,\,\, \alpha_t\circ \Psi_t=\Psi_t\circ \alpha\rq_t.$
\end{enumerate}
\end{defn}
\begin{rmk}
Suppose  $L^G_t$ and ${L\rq}_t^G$ are equivalent deformation. For every subgroup $H\leq G$, $H$-fixed point set $L^H$ is a Hom-Leibniz sub algebra. A formal equivariant isomorphism $\Psi_t$ induces formal isomophism $L^H[[t]]\to {L\rq}^H[[t]]$ for all subgroups $H$ of $G$.
\end{rmk}
From the second condition of the Definition (\ref{equ equivalent defn}) we have the following equations:
 \begin{align}
 &\sum_{i,j\geq 0}\psi_i(m\rq_j(x,y))t^{i+j}=\sum_{i,j,k\geq 0}m_i(\psi_j(x),\psi_k(y))t^{i+j+k},\\
 &\sum_{i,j \geq 0}\alpha_i(\psi_j(x))t^{i+j}=\sum_{i,j\geq 0}\psi_i(\alpha\rq_j(x))t^{i+j}.
 \end{align}
 Comparing coefficients of infinitesimals on both sides of the above equations, we have the following proposition.
 \begin{prop}
Equivariant infinitesimals of two equivalent equivariant deformations determine the same cohomology class.
  \end{prop}
  Similar to the non-equivariant case, we have the following rigidity theorem for equivariant deformations.
 \begin{theorem}
 Let $(G,L,[.,.],\alpha)$ be a Hom-Leibniz algebra equipped with an action of finite group $G$. If $\widetilde{HL}_G^2 (L, L)=0$ then $L$ is equivariantly rigid.
 \end{theorem}
 

\end{document}